\documentclass{amsart}
\usepackage{latexsym,amsmath,amssymb}
\newtheorem{thm}{Theorem}[section]
\newtheorem{cor}[thm]{Corollary}

\newtheorem{lem}[thm]{Lemma}
\newtheorem{prop}[thm]{Proposition}
\theoremstyle{definition}\newtheorem{defn}[thm]{Definition}
\theoremstyle{remark}

\numberwithin{equation}{section}

\begin{document}

\title[]
{Hyperbolic composition operators on Orlicz spaces}

\author{\sc\bf Y. Estaremi}
\address{\sc Y. Estaremi}
\email{y.estaremi@gu.ac.ir}

\address{Department of Mathematics, Faculty of Sciences, Golestan University, Gorgan, Iran.}

\thanks{}

\thanks{}



\date{}

\dedicatory{}

\commby{}

\begin{abstract}
 In the present paper we provide some equivalent conditions for composition operators to have shadowing property on Orlicz space $L^{\Phi}(\mu)$. Also, we obtain that for the composition operators on Orlicz spaces the notions of generalized hyperbolicity and the shadowing property coincide. The results of this paper extends similar results on $L^p$-spaces.
\end{abstract}

\maketitle

\section{ \sc\bf Introduction and Preliminaries}
%

In last two decades mathematicians have obtained plenty of interesting results in linear dynamics concerning dynamical properties such as transitivity, mixing, Li-Yorke and many others. Finding relations between expansivity, hyperbolicity,the shadoing property and structural stability is one of important questions in linear dynamics. Some classical results that are obtained on the relations between expansivity, hyperbolicity,the shadoing property and structural stability can be find in \cite{ad,eh,hed,ma,sm,wal}. In \cite{ddm} the authors have shown that for the class of bounded composition operators on $L^p(\mu)$, the
notion of generalized hyperbolicity and the shadowing property coincide. Indeed they provide some sufficient
and necessary conditions for composition operators to have the shadowing property.\\

 An Orlicz space $L^{\Phi}$ is known to be a natural generalization of $L^p$ spaces and has been considered
in various areas such as probability theory, partial differential equations, and mathematical finance.
 In this paper we extend these results for composition operators to Orlicz spaces. Indeed we provide some equivalent conditions for composition operators to have shadowing property on Orlicz space $L^{\Phi}(\mu)$. Also, we obtain that for the composition operators on Orlicz spaces the notions of generalized hyperbolicity and the shadowing property coincide.

Here for the convenience of the reader, we recall some essential facts on Orlicz spaces for later use. For more details on Orlicz spaces,  see \cite{kr,raor}.

A function $\Phi:\mathbb{R}\rightarrow [0,\infty]$ is called a \textit{Young's function}  if $\Phi$ is   convex, $\Phi(-x)=\Phi(x)$, $\Phi(0)=0$ and $\lim_{x\rightarrow \infty} \Phi(x)=+\infty$. With each \textit{Young's function} $\Phi$, one can associate another convex function $\Psi:\mathbb{R} \rightarrow [0, \infty]$ having similar properties defined by

$$\Psi(y)=\sup\{x\mid y\mid-\Phi(x): x\geq0\}, \ \ \ \ \ y\in \mathbb{R}.$$

The function $\Psi$ is called the \textit{complementary function} to $\Phi$. It follows from the definition that $\Psi(0)=0$, $\Psi(-y)=\Psi(y)$ and $\Psi(.)$ is a convex increasing function satisfying $\lim_{y\rightarrow \infty} \Psi(y)=+\infty$.

Let $\Phi$ be a \textit{Young's function}. Then we say $\Phi$ satisfies the
$\Delta_{2}$-condition or $\Phi$ is $\Delta_2$-regular, if $\Phi(2x)\leq
K\Phi(x) \; ( x\geq x_{0})$  for some constants
$K>0$ and $x_0>0$. Also, $\Phi$ satisfies the
$\Delta_{2}$-condition globally if $\Phi(2x)\leq
K\Phi(x) \; ( x\geq 0)$  for some
$K>0$. Also, $\Phi$ is said to satisfy the
$\Delta'$ condition, if $\exists c>0$
$(b>0)$ such that
$$\Phi(xy)\leq c\Phi(x)\Phi(y), \ \ \ x,y\geq x_{0}\geq 0.$$

If $x_{0}=0$, then it is said to hold
globally. If $\Phi\in \Delta'$, then $\Phi\in
\Delta_{2}$.\\

For a given complete $\sigma$-finite measure space$(X,\mathcal{F},\mu)$, let $L^0(\mathcal{F})$  be the linear space of  equivalence classes of $\mathcal{F}$-measurable real-valued functions on $X$, that is, we identify functions equal $\mu$-almost everywhere on $X$. The support $S(f)$ of a
measurable function $f$ is defined by $S(f):=\{x\in X : f(x)\neq
0\}$.

For a \textit{Young's function} $\Phi$, let $\rho_{\Phi}:L^{\Phi}(\mu)\rightarrow \mathbb{R}^{+}$ such that $\rho_{\Phi}(f)=\int_{X}\Phi(f)d\mu$ for all $f\in L^{\Phi}(\mu)$. Then the space
$$
L^{\Phi}(\mu)=\left\{f\in L^0(\mathcal{F}):\exists k>0, \rho_{\Phi}(kf)<\infty\right\}
$$
is called an Orlicz space. Define the functional

$$N_{\Phi}(f)=\inf \{k>0:\rho_{\Phi}(\frac{f}{k})\leq 1\}.$$

The functional $N_{\Phi}(.)$ is a norm on $L^{\Phi}(\mu)$ and is called \textit{guage norm}(or Luxemburge norm). Also, $(L^{\Phi}(\mu), N_{\Phi}(.))$  is a normed linear space. If almost every where equal functions are identified, then $(L^{\Phi}(\mu), N_{\Phi}(.))$ is a Banach space, the basic measure space $(X,\mathcal{F},\mu)$ is unrestricted. Hence every element of $L^{\Phi}(\mu)$ is a class of measurable functions that are almost every where equal. Also, there is another norm on $L^{\Phi}(\mu)$, defined as follow:
\begin{equation}\label{e1}
\|f\|_{\Phi}=\sup\{\int_{X}\mid fg\mid d\mu: g\in B_{\Psi}\}=\sup\{\mid\int_{X}fg d\mu\mid: g\in B_{\Psi}\},
\end{equation}
in which $B_{\Psi}=\{g\in L^{\Psi}(\mu): \int_{X}\Psi(\mid g\mid )d\mu\leq 1\}$.
  The norm $\|.\|_{\Phi}$ is called \textit{Orlicz norm}. For any $f\in L^{\Phi}(\mu)$, $\Phi$ being a Young function, we have\\
\begin{equation}\label{e2}
N_{\Phi}(f)\leq \|f\|_{\Phi}\leq2N_{\Phi}(f).
\end{equation}
For every $F\in\mathcal{F}$ with $0<\mu(F)<\infty$, we have $N_{\Phi}(\chi_F)=\frac{1}{\Phi^{-1}(\frac{1}{\mu(F)})}$.\\

Throughout this paper, $(X,\mathcal{F},\mu)$ will be a measure space, that is, $X$ is a set, $\mathcal{F}$ is a sigma algebra on $X$ and $\mu$ is a positive measure on $\mathcal{F}$. Also, we assume that $\varphi:X\rightarrow X$ is a non-singular measurable transformation, that is, $\varphi^{-1}(F)\in \mathcal{F}$, for every $F\in \mathcal{F}$ and $\mu(\varphi^{-1}(F))=0$, if $\mu(F)=0$. Moreover, for the non-singular measurable transformation $\varphi:X\rightarrow X$, if there exists a positive constant $c$ for which
\begin{equation}\label{ec}
\mu(\varphi^{-1}(F))\leq c\mu(F), \ \ \ \ \text{for every} F\in \mathcal{F},
\end{equation}
then the linear operator
$$C_{\varphi}:L^{\Phi}(\mu)\rightarrow L^{\Phi}(\mu), \ \  \ \ \ \ C_{\varphi}(f)=f\circ\varphi,$$
is well-defined and continuous on the Orlicz space $L^{\Phi}(\mu)$ and is called composition operator. For more details on composition operators on Orlicz spaces one can refer to \cite{chkm}.

\section{\sc\bf Hyperbolicity and shadowing property}

In this section first we rewrite the definitions of wandering set, dissipative system and dissipative system of bounded distortion in Orlicz spaces, that were defined for $L_p$ spaces in \cite{ddm}.

\begin{defn}
Let $(X, \mathcal{F}, \mu)$ be a measure space and $\varphi:X \rightarrow X$ be an invertible non-singular transformation. A measurable set $W\subseteq X$ is called a wandering set for $\varphi$ if the sets $\{\varphi^{-n}(W)\}_{n\in\mathbb{Z}}$ are pair-wise disjoint.
\end{defn}

\begin{defn}
Let $(X, \mathcal{F}, \mu)$ be a measure space and $\varphi:X\rightarrow X$ be an invertible non-singular transformation. The quadruple $(X, \mathcal{F}, \mu, \varphi)$ is called
\begin{itemize}
  \item a dissipative system generated by $W$, if $X=\dot{\cup}_{k\in \mathbb{Z}}\varphi^k(W)$ for some $W\in \mathcal{F}$ with $0<\mu(W)<\infty$ (the symbol $\dot{\cup}$ denotes pairwise disjoint union);
  \item a dissipative system, of bounded distortion, generated by $W$, if there exists $K>0$, such that
  \begin{equation}\label{edis}
  \frac{1}{K}N_{\Phi}(C^k_{\varphi}(\chi_W))N_{\Phi}(\chi_F)\leq N_{\Phi}(C^k_{\varphi}(\chi_F))N_{\Phi}(\chi_W)\leq KN_{\Phi}(C^k_{\varphi}(\chi_W))N_{\Phi}(\chi_F),
  \end{equation}
  for all $k\in \mathbb{Z}$ and $F\in \mathcal{F}_W=\{F\cap W: F\in \mathcal{F}\}$. If we replace $N_{\Phi}$ by the norm of $L_p$, then we will have the bounded distortion property in $L_p$-spaces.(Definition 2.6.4, \cite{ddm})
 \end{itemize}
\end{defn}
\begin{defn}
A composition dynamical system $(X, \mathcal{F}, \mu, \varphi, C_{\varphi})$ is called
\begin{itemize}
\item dissipative composition dynamical system, generated by $W$, if $(X, \mathcal{F}, \mu, \varphi)$ is a dissipative system generated by $W$;
\item dissipative composition dynamical system, of bounded distortion, generated by $W$, if $(X, \mathcal{F}, \mu, \varphi)$ is a dissipative system of bounded distortion, generated by $W$.
\end{itemize}

\end{defn}
In the following we have a proposition in Orlicz spaces similar to the Proposition 2.6.5 of \cite{ddm}.
\begin{prop}\label{p2.6}
Let $(X, \mathcal{F}, \mu, \varphi)$ be a dissipative system of bounded distortion, generated by $W$. Then the followings hold:
\begin{enumerate}
  \item  There exists $H>0$ such that
\begin{equation}\label{gedis}
\frac{1}{H}\frac{N_{\Phi}(C^{t+s}_{\varphi}(\chi_W))}{N_{\Phi}(C^s_{\varphi}(\chi_W))}\leq \frac{N_{\Phi}(C^{t+s}_{\varphi}(\chi_F))}{N_{\Phi}(C^s_{\varphi}(\chi_F))}\leq H\frac{N_{\Phi}(C^{t+s}_{\varphi}(\chi_W))}{N_{\Phi}(C^s_{\varphi}(\chi_W))},
\end{equation}
for all $F\in \mathcal{F}_W$ with $\mu(F)>0$ and all $s,t\in \mathbb{Z}$.
\item Let $\Phi\in \Delta_2$. If
$$\sup\{\frac{N_{\Phi}(C^{k-1}_{\varphi}(\chi_W))}{{N_{\Phi}(C^{k}_{\varphi}(\chi_W))}}, \frac{N_{\Phi}(C^{k+1}_{\varphi}(\chi_W))}{{N_{\Phi}(C^{k}_{\varphi}(\chi_W))}}: k\in \mathbb{Z}\}$$

is finite and $\varphi$ is bijective, then $\varphi$ and $\varphi^{-1}$ satisfy \ref{ec} condition.
\end{enumerate}
\end{prop}
\begin{proof}
It is clear from the condition \ref{edis} that $N_{\Phi}(\chi_F)=0$ if and only if $N_{\Phi}(C^k_{\varphi}(\chi_F))=N_{\Phi}(\chi_{\varphi^{-k}(F)})=0$, for all $k\in \mathbb{Z}$ and therefore, $\mu(F)=0$ if and only if $\mu(\varphi^k(F))=0$, for all $k\in \mathbb{Z}$. This implies that the condition \ref{gedis} is well-defined. By using the condition \ref{edis}, for every $s,t\in \mathcal{Z}$ we have
\begin{align*}
\frac{N_{\Phi}(C^{t+s}_{\varphi}(\chi_W))}{N_{\Phi}(C^s_{\varphi}(\chi_W))}&=\frac{N_{\Phi}(C^{t+s}_{\varphi}(\chi_W))}{N_{\Phi}(\chi_W)}\frac{N_{\Phi}(\chi_W)}{N_{\Phi}(C^s_{\varphi}(\chi_W))}\\
&\leq K^2 \frac{N_{\Phi}(C^{t+s}_{\varphi}(\chi_F))}{N_{\Phi}(\chi_F)}\frac{N_{\Phi}(\chi_F)}{N_{\Phi}(C^s_{\varphi}(\chi_F))}\\
&=K^2\frac{N_{\Phi}(C^{t+s}_{\varphi}(\chi_F))}{N_{\Phi}(C^s_{\varphi}(\chi_F))},
\end{align*}
and by a similar way we have
 $$\frac{N_{\Phi}(C^{t+s}_{\varphi}(\chi_W))}{N_{\Phi}(C^s_{\varphi}(\chi_W))}\geq\frac{1}{K^2}\frac{N_{\Phi}(C^{t+s}_{\varphi}(\chi_F))}{N_{\Phi}(C^s_{\varphi}(\chi_F))}.$$
 If we put $H=K^2$, then the proof is complete.\\
 (2) Let $$M=\sup\{\frac{N_{\Phi}(C^{k-1}_{\varphi}(\chi_W))}{{N_{\Phi}(C^{k}_{\varphi}(\chi_W))}}, \frac{N_{\Phi}(C^{k+1}_{\varphi}(\chi_W))}{{N_{\Phi}(C^{k}_{\varphi}(\chi_W))}}: k\in \mathbb{Z}\}.$$
 For each $B\in\mathcal{F}$, we have $B=\cup^{+\infty}_{k=-\infty} \varphi^k(W)\cap B$. Let $B_k=\varphi^k(W)\cap B$. As $B_k$'s are disjoint and the measure $mu$ is countably additive, we just prove the condition \ref{ec} for $B_k$. If $\mu(B_k)=0$, then $N_{\Phi}(\chi_{B_k})=0$ and if we set $F=\varphi^{-k}(B_k)$, then $C^{-k}_{\varphi}(\chi_{\varphi^{-k}(B_{k})})=\chi_{B_k}$ and so by the condition \ref{edis} as:
 $$
  \frac{1}{K}N_{\Phi}(C^{-k}_{\varphi}(\chi_W))N_{\Phi}(\chi_F)\leq N_{\Phi}(C^{-k}_{\varphi}(\chi_F))N_{\Phi}(\chi_W)\leq KN_{\Phi}(C^{-k}_{\varphi}(\chi_W))N_{\Phi}(\chi_F),$$
  we get that $N_{\Phi}(\chi_{\varphi^{-k}(B_k)})=0$ and so $\chi_{\varphi^{-k}(B_k)}=0$. Thus $\mu(\varphi^{-k}(B_k))=0$ and therefore $\mu(\varphi^{-1}(B_k))=0$. For the case $\mu(B_k)>0$, we apply the condition \ref{gedis} to $F=\varphi{B_k}$ we get that
  $$\frac{N_{\Phi}(C_{\varphi}(\chi_{B_k}))}{N_{\Phi}(\chi_{B_k})}\leq H\frac{N_{\Phi}(C^{-1
  +k}_{\varphi}(\chi_W))}{N_{\Phi}(C^{k}_{\varphi}(\chi_W))}\leq HM.$$
  If we let $c=HM$, then
  $$N_{\Phi}(C_{\varphi}(\chi_{B_k}))\leq cN_{\Phi}(\chi_{B_k})$$
  and so
  $$\Phi^{-1}(\frac{1}{\mu(B_k)})\leq c\Phi^{-1}(\frac{1}{\mu(\varphi^{-1}(B_k))}).$$
  Since $\Phi\in \Delta_2$ and $\Phi$ is increasing, then there exists $n>0$ such that
  $$\frac{1}{\mu(B_k)}=\Phi\Phi^{-1}(\frac{1}{\mu(B_k)})\leq \Phi( c\Phi^{-1}(\frac{1}{\mu(\varphi^{-1}(B_k))}))\leq \frac{cK^n}{2^n}\frac{1}{\mu(\varphi^{-1}(B_k))}.$$
  Now by taking $d=\frac{cK^n}{2^n}$ we have $\mu(\varphi^{-1}(B_k))\leq d\mu(B_k)$. So the condition \ref{ec} holds for $\varphi$. Similarly we have for $\varphi^{-1}$.

\end{proof}
Let $h_k=\frac{d(\mu\circ\varphi^{-k})}{d\mu}$ denotes the Radon-Nikodym derivative of $\mu\circ\varphi^{-k}$ with respect to $\mu$, where
$\mu\circ\varphi^{-k}(F)=\mu(\varphi^{-k}(F))$, for every $F\in \mathcal{F}$.
\begin{prop}
Let $(X, \mathcal{F}, \mu, \varphi)$ be a dissipative system generated by $W$ and $\varphi$ is injective. We put $m_k={ess\,inf}_{x\in W} h_k(x)$, and $M_k={ess\,sup}_{x\in W} h_k(x)$. Without lose of generality we can assume that $m_k\leq 1$ and $M_k\geq 1$. If the sequence $\{\frac{M_k}{m_k}\}_{k\in \mathbb{Z}}$ is bounded, then $\varphi$ is of bounded distortion on $W$.
\end{prop}
\begin{proof}
Suppose that $\{\frac{M_k}{m_k}\}_{k\in \mathbb{Z}}$ is bounded, so there exists $K>0$ such that $\frac{M_k}{m_k}\leq K$. Let $F\in \mathcal{F}$ with $\mu(F)>0$. Then for every $k\in \mathbb{Z}$,
\begin{align*}
\int_X\Phi(\frac{C^k_{\varphi}(\chi_{F})}{M_kN_{\Phi}(\chi_F)}d\mu&\leq \int_X\Phi(\frac{\chi_{F}}{M_kN_{\Phi}(\chi_F)}h_kd\mu\\
&\leq \int_X\Phi(\frac{\chi_{F}}{M_kN_{\Phi}(\chi_F)}M_kd\mu\\
&\int_X\Phi(\frac{\chi_{F}}{N_{\Phi}(\chi_F)}d\mu\leq 1.
\end{align*}
This implies that $$N_{\Phi}(C^k_{\varphi}(\chi_{F}))\leq M_kN_{\Phi}(\chi_F).$$
Also,
\begin{align*}
\int_X\Phi(\frac{m_k\chi_{F}}{N_{\Phi}(C^k_{\varphi}(\chi_F))}d\mu&\leq \int_X\Phi(\frac{m_k\chi_{F}}{N_{\Phi}(C^k_{\varphi}(\chi_F))}\frac{h_k}{m_k}d\mu\\
&\int_X\Phi(\frac{m_kC^k_{\varphi}(\chi_{F})}{m_kN_{\Phi}(C^k_{\varphi}(\chi_F))}d\mu\leq 1.
\end{align*}
This implies that
 $$m_kN_{\Phi}(\chi_W)=N_{\Phi}(m_k\chi_W)\leq N_{\Phi}(C^k_{\varphi}(\chi_W)).$$
 By the above observations we have
 $$\frac{1}{K}\frac{N_{\Phi}(C^{k}_{\varphi}(\chi_W))}{N_{\Phi}(\chi_W)}\leq\frac{N_{\Phi}(C^{k}_{\varphi}(\chi_F))}{N_{\Phi}(\chi_F)}\leq K\frac{N_{\Phi}(C^{k}_{\varphi}(\chi_W))}{N_{\Phi}(\chi_W)}.$$
 So we get the result.
\end{proof}

\begin{defn} Let $\varphi,\varphi^{-1}$ satisfy the condition \ref{ec} and $(X,\mathcal{F},\mu,\varphi)$ be a dissipative system generated by $W$. We define the following conditions
\begin{equation}\label{hc}
 \overline{\lim}_{n\rightarrow\infty}\sup_{k\in\mathbb{Z}}\left(\frac{\Phi^{-1}(\frac{1}{\mu(\varphi^{k+n}(W))})}{\Phi^{-1}(\frac{1}{\mu(\varphi^{k}(W))})}\right)^{\frac{1}{n}}<1 \end{equation}
  \begin{equation}\label{hd} \underline{\lim}_{n\rightarrow\infty}\inf_{k\in\mathbb{Z}}\left(\frac{\Phi^{-1}(\frac{1}{\mu(\varphi^{k+n}(W))})}{\Phi^{-1}(\frac{1}{\mu(\varphi^{k}(W))})}\right)^{\frac{1}{n}}>1 \end{equation}
\begin{equation}\label{gh}
\overline{\lim}_{n\rightarrow\infty}\sup_{k\in-\mathbb{N}_0}\left(\frac{\Phi^{-1}(\frac{1}{\mu(\varphi^{k}(W))})}{\Phi^{-1}(\frac{1}{\mu(\varphi^{k-n}(W))})}\right)^{\frac{1}{n}}<1
 \ \ \text{and} \ \ \underline{\lim}_{n\rightarrow\infty}\inf_{k\in\mathbb{N}_0}\left(\frac{\Phi^{-1}(\frac{1}{\mu(\varphi^{k}(W))})}{\Phi^{-1}(\frac{1}{\mu(\varphi^{k+n}(W))})}\right)^{\frac{1}{n}}>1 \end{equation}
\end{defn}
From now on we assume that $(X,\mathcal{F},\mu, \varphi)$ is a dissipative system generated by $W$ such that the associated composition operator $C_{\varphi}$ is an invertible operator on Orlicz space $L^{\Phi}(\mu)$. Now we rewrite some notations and terminologies for Orlicz spaces similar to those one that were introduced for $L^p(\mu)$-spaces in \cite{ddm}.
\begin{defn}
Let $f\in L^{\Phi}(\mu)$. Then we define $f=f_+ +f_-$, where
$$
f_+(x)=\left\{
  \begin{array}{ll}
    0, & \hbox{if $x\in\cup^{\infty}_{k=0}\varphi^k(W)$} \\
    f(x), & \hbox{otherwise,} \\
     \end{array}
\right.
$$
and,
$$\ \ \ \ f_-(x)=\left\{
  \begin{array}{ll}
    0, & \hbox{if $x\in\cup^{\infty}_{k=1}\varphi^{-k}(W)$} \\
    f(x), & \hbox{otherwise.} \\
     \end{array}
\right.$$
 Let $L^{\Phi}_+=\{f_+:f\in L^{\Phi}(\mu)\}$ and $L^{\Phi}_-=\{f_-:f\in L^{\Phi}(\mu)\}$. It is clear that $L^{\Phi}(\mu)=L^{\Phi}_+\oplus L^{\Phi}_-$ and $C_{\varphi}(L^{\Phi}_+)\subseteq L^{\Phi}_+$ and $C^{-1}_{\varphi}(L^{\Phi}_-)\subseteq L^{\Phi}_-$.
\end{defn}
\begin{defn}
Let $K,t>0$ and let $\mathcal{U}^{\Phi}_C(K,t)$ and $\mathcal{U}^{\Phi}_D(K,t)$ be the set of all $f\in L^{\Phi}(\mu)$ which satisfy the following conditions, respectively:
\begin{equation}\label{u_c}
\sup_{k\in \mathbb{Z}}\frac{N_{\Phi}(C^k_{\varphi^{-1}}f)}{N_{\Phi}(C^{k+n}_{\varphi^{-1}}(f))}\leq Kt^n, \ \ \ \ \ \forall n\in \mathbb{N}
\end{equation}
and
\begin{equation}\label{u_d}
\inf_{k\in \mathbb{Z}}\frac{N_{\Phi}(C^k_{\varphi^{-1}}f)}{N_{\Phi}(C^{k+n}_{\varphi^{-1}}(f))}\geq Kt^n, \ \ \ \ \ \forall n\in \mathbb{N}.
\end{equation}
Also, we let $\mathcal{U}^{\Phi}_{GH+}(K,t)$ and $\mathcal{U}^{\Phi}_{GH-}(K,t)$ the set of all $f$ in $L^{\Phi}_+$ and $L^{\Phi}_-$, respectively, which satisfy the following conditions:
\begin{equation}\label{ugh+}
\sup_{k\in -\mathbb{N}_0}\frac{N_{\Phi}(C^{k-n}_{\varphi^{-1}}(f))}{N_{\Phi}(C^{k}_{\varphi^{-1}}(f))}\leq Kt^n, \ \ \ \ \ \forall n\in \mathbb{N}
\end{equation}
and
\begin{equation}\label{ugh-}
\inf_{k\in \mathbb{N}_0}\frac{N_{\Phi}(C^k_{\varphi^{-1}}(f))}{N_{\Phi}(C^{k+n}_{\varphi^{-1}}(f))}\geq K\frac{1}{t^n}, \ \ \ \ \ \forall n\in \mathbb{N}.
\end{equation}
\end{defn}
Here we recall a simple fact about the definition of upper and lower limits.
\begin{prop}\label{p2.9}
Let $\{a_n\}_{n\in \mathbb{N}}$ be a sequence of non-negative real numbers and $t>0$. Then the following hold:
\begin{enumerate}
  \item If $\overline{\lim}_{n\rightarrow \infty}a^{\frac{1}{n}}_n<t$, then there exists $K>0$ such that $a_n\leq Kt^n$, for every $n\in \mathbb{N}$.
  \item If $\underline{\lim}_{n\rightarrow \infty}a^{\frac{1}{n}}_n>t$, then there exists $K>0$ such that $a_n\geq Kt^n$, for every $n\in \mathbb{N}$.
\end{enumerate}
\end{prop}
\begin{prop}\label{pp2.10}
The following hold.
\begin{enumerate}
  \item The condition \ref{hc} holds if and only if $\chi_{w}\in \mathcal{U}^{\Phi}_{C}(K,t)$ for some $K>0$ and $t<1$.
  \item The condition \ref{hd} holds if and only if $\chi_{w}\in \mathcal{U}^{\Phi}_{D}(K,t)$ for some $K>0$ and $t>1$.
  \item The condition \ref{gh} holds if and only if there exist $K>0$ and $t<1$ such that $\chi_{w}\in \mathcal{U}^{\Phi}_{GH-}(K,t)$ and $C_{\varphi}(\chi_{w})\in \mathcal{U}^{\Phi}_{GH+}(K,t)$.
\end{enumerate}
\end{prop}
\begin{proof}
(1) Suppose that the condition \ref{hc} holds. Let $a_n=\sup_{k\in \mathbb{Z}}\frac{C^{k}_{\varphi^{-1}}(\chi_W)}{C^{k+n}_{\varphi^{-1}}(\chi_W)}$. Since $N_{\Phi}(C^{k}_{\varphi^{-1}}(\chi_W))=\frac{1}{\Phi^{-1}(\frac{1}{\mu(\varphi^k(W))})}$, then by definition of the condition \ref{hc} and the Proposition \ref{p2.9} we get that there exist $K>0$ and $t<1$ such that $\chi_w\in \mathcal{U}^{\Phi}_{C}(K,t)$.\\
For the converse,  Let $K>0$ and $t<1$ such that $\chi_w\in \mathcal{U}^{\Phi}_{C}(K,t)$, i.e.,
$$\sup_{k\in \mathbb{Z}}\frac{\Phi^{-1}(\frac{1}{\mu(\varphi^{k+n}(W))})}{\Phi^{-1}(\frac{1}{\mu(\varphi^k(W))})}=\sup_{k\in \mathbb{Z}}\frac{N_{\Phi}(C^k_{\varphi^{-1}}(\chi_W))}{N_{\Phi}(C^{k+n}_{\varphi^{-1}}(\chi_W))}\leq Kt^n.$$
Hence
\begin{align*}
\overline{\lim}_{n\rightarrow \infty}\sup_{k\in \mathbb{Z}}\left(\frac{\Phi^{-1}(\frac{1}{\mu(\varphi^{k+n}(W))})}{\Phi^{-1}(\frac{1}{\mu(\varphi^k(W))})}\right)^{\frac{1}{n}}&=\overline{\lim}_{n\rightarrow \infty}\sup_{k\in \mathbb{Z}}\left(\frac{N_{\Phi}(C^k_{\varphi^{-1}}(\chi_W))}{N_{\Phi}(C^{k+n}_{\varphi^{-1}}(\chi_W))}\right)^{\frac{1}{n}}\\
&\leq\overline{\lim}_{n\rightarrow \infty} K^{\frac{1}{n}}t\\
&=t<1.
\end{align*}
So the condition \ref{hc} holds.\\
(2) By using the definition of the condition \ref{hd} and Proposition \ref{p2.9}, similar to the proof of (1) we get the proof.\\
(3) Suppose that the condition \ref{gh} holds. We replace $W$ by $\varphi^{-1}(W)$ in the first part of condition \ref{gh}. Then by Proposition \ref{p2.9} there exist $K>0$ and $0<t<1$ such that
$$
\sup_{k\in-\mathbb{N}_0}\left(\frac{\Phi^{-1}(\frac{1}{\mu(\varphi^{k}(\varphi^{-1}(W)))})}{\Phi^{-1}(\frac{1}{\mu(\varphi^{k-n}(\varphi^{-1}(W)))})}\right)\leq Kt^n$$
and
$$\inf_{k\in\mathbb{N}_0}\left(\frac{\Phi^{-1}(\frac{1}{\mu(\varphi^{k}(W))})}{\Phi^{-1}(\frac{1}{\mu(\varphi^{k+n}(W))})}\right)\geq K\frac{1}{t^n},$$
 for every $n\in \mathbb{N}$.
Hence
$$
\sup_{k\in -\mathbb{N}_0}\frac{N_{\Phi}(C^{k-n}_{\varphi^{-1}}(\chi_{\varphi^{-1}(W)}))}{N_{\Phi}(C^{k}_{\varphi^{-1}}(\chi_{\varphi^{-1}(W)}))}=\sup_{k\in-\mathbb{N}_0}\left(\frac{\Phi^{-1}(\frac{1}{\mu(\varphi^{k}(\varphi^{-1}(W)))})}{\Phi^{-1}(\frac{1}{\mu(\varphi^{k-n}(\varphi^{-1}(W)))})}\right)\leq Kt^n,$$
and
$$
\inf_{k\in \mathbb{N}_0}\frac{N_{\Phi}(C^k_{\varphi^{-1}}(\chi_W))}{N_{\Phi}(C^{k+n}_{\varphi^{-1}}(\chi_W))}=\inf_{k\in\mathbb{N}_0}\left(\frac{\Phi^{-1}(\frac{1}{\mu(\varphi^{k}(W))})}{\Phi^{-1}(\frac{1}{\mu(\varphi^{k+n}(W))})}\right)\geq
K\frac{1}{t^n}$$
So we get that $\chi_W\in \mathcal{U}^{\Phi}_{GH-}(K,t)$ and $C_{\varphi}(\chi_W)\in \mathcal{U}^{\Phi}_{GH+}(K,t)$.\\
For the converse, similar to proof of (1), we can get the proof.

\end{proof}
\begin{prop}\label{pp2.11}
The following statements are true.
\begin{itemize}
  \item Let $\mathcal{U}^{\Phi}(K,t)\in \{\mathcal{U}^{\Phi}_C(K,t), \mathcal{U}^{\Phi}_D(K,t)\}$. If $f\in \mathcal{U}^{\Phi}(K,t)$, then $C^j_{\varphi}(f)\in \mathcal{U}^{\Phi}(K,t)$, for all $j\in \mathbb{Z}$.
  \item If $f\in \mathcal{U}^{\Phi}_{GH+}(K,t)$, then $C^j_{\varphi}(f)\in \mathcal{U}^{\Phi}_{GH+}(K,t)$, for all $j\geq0$.
  \item If $f\in\mathcal{U}^{\Phi}_{GH-}(K,t)$, then $C^j_{\varphi}(f)\in \mathcal{U}^{\Phi}_{GH-}(K,t)$, for all $j\leq0$.
\end{itemize}
\end{prop}
\begin{proof}
By a straightforward calculations we can get all statements by definitions.
\end{proof}
\begin{prop}\label{p2.11}
Let $\mathcal{U}^{\Phi}(K,t)\in \{\mathcal{U}^{\Phi}_{C}(K,t),\mathcal{U}^{\Phi}_{D}(K,t),\mathcal{U}^{\Phi}_{GH+}(K,t),\mathcal{U}^{\Phi}_{GH-}(K,t)\}$. Then the following are true.
\begin{enumerate}
  \item If $f\in \mathcal{U}^{\Phi}(K,t)$ and $\alpha\in\mathbb{C}\setminus\{0\}$, then $\alpha.f\in \mathcal{U}^{\Phi}(K,t)$.
  \item If $f_1, f_2\in \mathcal{U}^{\Phi}(K,t)$ with disjoint supports, then $f_1+f_2\in \mathcal{U}^{\Phi}(K,t)$
\end{enumerate}
\end{prop}
\begin{proof}
(1) It is clear by definition.\\
(2) Let $\mathcal{U}^{\Phi}(K,t)=\mathcal{U}^{\Phi}_{C}(K,t)$ and $f_1, f_2\in \mathcal{U}^{\Phi}_C(K,t)$ be with disjoint supports. So by definition we have
$$\sup_{k\in \mathbb{Z}}\frac{N_{\Phi}(C^k_{\varphi^{-1}}(f_i))}{N_{\Phi}(C^{k+n}_{\varphi^{-1}}(f_i))}\leq Kt^n, \ \ \ \ \ \forall n\in \mathbb{N},$$
for $i=1,2$. Hence
$$N_{\Phi}(C^k_{\varphi^{-1}}(f_i))\leq Kt^n N_{\Phi}(C^{k+n}_{\varphi^{-1}}(f_i)),$$
for every $k\in \mathbb{Z}$ and $i=1,2$. Since $f_1$ and $f_2$ have disjoint supports, then by definition of $N_{\Phi}(.)$ it is clear that $N_{\Phi}(f_1)+N_{\Phi}(f_2)\leq2N_{\Phi}(f_1+f_2)$. Moreover, Since $f_1$ and $f_2$ have disjoint supports, then for each $k\in \mathbb{Z}$, $f_1\circ\varphi^{-k}$ and $f_2\circ \varphi^{-k}$ have disjoint support too. By these observations for every $k\in \mathbb{Z}$ we have
\begin{align*}
N_{\Phi}(C^{k}_{\varphi^{-1}}(f_1+f_2))&=N_{\Phi}((f_1+f_2)\circ\varphi^{-k})\\
&=N_{\Phi}((f_1)\circ\varphi^{-k})+N_{\Phi}((f_2)\circ\varphi^{-k})\\
&\leq N_{\Phi}(C^k_{\varphi^{-1}}(f_1))+N_{\Phi}(C^k_{\varphi^{-1}}(f_2))\\
&\leq Kt^n (N_{\Phi}(C^{k+n}_{\varphi^{-1}}(f_1))+ N_{\Phi}(C^{k+n}_{\varphi^{-1}}(f_2)))\\
&\leq 2Kt^nN_{\Phi}(C^{k+n}_{\varphi^{-1}}(f_1+f_2)).
\end{align*}
Therefore
$$\sup_{k\in \mathbb{Z}}\frac{N_{\Phi}(C^k_{\varphi^{-1}}(f_1+f_2))}{N_{\Phi}(C^{k+n}_{\varphi^{-1}}(f_1+f_2))}\leq 2Kt^n, \ \ \ \ \ \forall n\in \mathbb{N}.$$
This implies that $f_1+f_2\in\mathcal{U}^{\Phi}_C(K,t)$. For other cases the proof is similar to the case $\mathcal{U}^{\Phi}_{C}(K,t)$.
\end{proof}

\begin{prop}\label{p2.13} Let $(X, \mathcal{F}, \mu, \varphi)$ be with bounded distortion and $H$ be the bounded distortion constant from Proposition \ref{p2.6}. Then for each $j\in \mathbb{Z}$ the following statements hold.
\begin{enumerate}
  \item Let $\mathcal{U}^{\Phi}(K,t)\in \{\mathcal{U}^{\Phi}_C(K,t), \mathcal{U}^{\Phi}_D(K,t)\}$. If $C^{-j}_{\varphi}(\chi_W)\in \mathcal{U}^{\Phi}(K,t)$, then $C^{-j}_{\varphi}(\chi_F)\in \mathcal{U}^{\Phi}(HK,t)$, for all $F\subseteq W$ with $\mu(F)>0$.
  \item If $C^{-j}_{\varphi}(\chi_W)\in \mathcal{U}^{\Phi}_{GH-}(K,t)$, for $j\geq0$, then $C^{-j}_{\varphi}(\chi_F)\in \mathcal{U}^{\Phi}_{GH-}(HK,t)$, for all $F\subseteq W$ with $\mu(F)>0$.
  \item If $C^{-j}_{\varphi}(\chi_W)\in \mathcal{U}^{\Phi}_{GH+}(K,t)$, for $j<0$, then $C^{-j}_{\varphi}(\chi_F)\in \mathcal{U}^{\Phi}_{GH+}(HK,t)$, for all $F\subseteq W$ with $\mu(F)>0$.
\end{enumerate}
\end{prop}
\begin{proof}
(1) Let $\mathcal{U}^{\Phi}(K,t)=\mathcal{U}^{\Phi}_C(K,t)$, $C^{-j}_{\varphi}(\chi_W)\in \mathcal{U}^{\Phi}(K,t)$ and $F\subseteq W$ with $\mu(F)>0$. We note that for each $k\in \mathbb{Z}$, $C^k_{\varphi^{-1}}=C^{-k}_{\varphi}$.
By the proposition \ref{p2.6}, there exists $0<H<\infty$ such that for all $t,s\in \mathbb{Z}$,
$$
\frac{1}{H}\frac{N_{\Phi}(C^{t+s}_{\varphi}(\chi_W))}{N_{\Phi}(C^s_{\varphi}(\chi_W))}\leq \frac{N_{\Phi}(C^{t+s}_{\varphi}(\chi_F))}{N_{\Phi}(C^s_{\varphi}(\chi_F))}\leq H\frac{N_{\Phi}(C^{t+s}_{\varphi}(\chi_W))}{N_{\Phi}(C^s_{\varphi(\chi_W)})}.
$$
So for a fixed $n\in \mathbb{N}$, we have
$$ \frac{N_{\Phi}(C^{-(k+j)}_{\varphi}(\chi_F))}{N_{\Phi}(C^{-(k+j+n)}_{\varphi}(\chi_F))}\leq H\frac{N_{\Phi}(C^{-(k+j)}_{\varphi}(\chi_W))}{N_{\Phi}(C^{-(k+j+n)}_{\varphi(\chi_W)})},$$
Hence
$$\sup_{k\in\mathbb{Z}}\left(\frac{N_{\Phi}(C^{-(k+j)}_{\varphi}(\chi_F))}{N_{\Phi}(C^{-(k+j+n)}_{\varphi}(\chi_F))}\right)\leq H\sup_{k\in \mathbb{Z}}\left(\frac{N_{\Phi}(C^{-(k+j)}_{\varphi}(\chi_W))}{N_{\Phi}(C^{-(k+j+n)}_{\varphi}(\chi_W))}\right),$$

Moreover, since $C^{-j}_{\varphi}(\chi_W)\in \mathcal{U}^{\Phi}(K,t)$, then by definition we have
$$\sup_{k\in \mathbb{Z}}\left(\frac{N_{\Phi}(C^{k+j}_{\varphi^{-1}}(\chi_W))}{N_{\Phi}(C^{k+j+n}_{\varphi^{-1}}(\chi_W))}\right)\leq Kt^n.$$
Therefore
$$\sup_{k\in\mathbb{Z}}\left(\frac{N_{\Phi}(C^{k}_{\varphi^{-1}}(\chi_{\varphi^j(F)}))}{N_{\Phi}(C^{k+n}_{\varphi^{-1}}(\chi_{\varphi^j(F)}))}\right)=\sup_{k\in\mathbb{Z}}\left(\frac{N_{\Phi}(C^{k+j}_{\varphi^{-1}}(\chi_F))}{N_{\Phi}(C^{k+j+n}_{\varphi^{-1}}(\chi_F))}\right)\leq HKt^n,$$
this means that  $C^{-j}_{\varphi}(\chi_F)\in \mathcal{U}^{\Phi}(HK,t)$. Proofs of other cases are similar.
\end{proof}
\begin{prop}\label{p2.14}
If $\Phi\in\Delta_2$, then the following statements are true.
\begin{enumerate}
  \item The set $$\{\sum^n_{i=0}a_i\chi_{F_i}: a_i\in \mathbb{C}, F_i\subseteq\varphi^{j_i}(W), j_i\in\mathbb{Z}, \mu(F_i)>0, F_i\cap F_{i'}=\emptyset, i\neq i'\}$$
  is dense in $L^{\Phi}(\mu)$.
  \item The set $$\{\sum^n_{i=0}a_i\chi_{F_i}: a_i\in \mathbb{C}, F_i\subseteq\varphi^{j_i}(W), j_i<0, \mu(F_i)>0, F_i\cap F_{i'}=\emptyset, i\neq i'\}$$
  is dense in $L^{\Phi}_+$.
  \item The set $$\{\sum^n_{i=0}a_i\chi_{F_i}: a_i\in \mathbb{C}, F_i\subseteq\varphi^{j_i}(W), j_i<0, \mu(F_i)>0, F_i\cap F_{i'}=\emptyset, i\neq i'\}$$
  is dense in $L^{\Phi}_-$.
\end{enumerate}
\end{prop}
\begin{proof}
Since $\Phi\in \Delta_2$, then the set of simple functions is dense in $L^{\Phi}(\mu)$ and easily we get the proof.
\end{proof}
\begin{prop}\label{p2.15}
Suppose that $\Phi\in \Delta_2$ and $(X,\mathcal{F}, \mu, \varphi)$ has bounded distortion.
\begin{enumerate}
  \item Let $\mathcal{U}^{\Phi}(K,t)\in \{\mathcal{U}^{\Phi}_C(K,t), \mathcal{U}^{\Phi}_D(K,t)\}$. If $\chi_W\in\mathcal{U}^{\Phi}(K,t)$, then $\mathcal{U}^{\Phi}(HK,t)=L^{\Phi}(\mu)$.
  \item If $\chi_W\in\mathcal{U}^{\Phi}_{GH-}(K,t)$, then $\mathcal{U}^{\Phi}_{GH-}(HK,t)=L^{\Phi}_-$.
  \item If $C_{\varphi}(\chi_W)\in\mathcal{U}^{\Phi}_{GH+}(K,t)$, then $\mathcal{U}^{\Phi}_{GH+}(HK,t)=L^{\Phi}_+$.
\end{enumerate}
\end{prop}
\begin{proof}
(1)  Let $\mathcal{U}^{\Phi}(K,t)\in \{\mathcal{U}^{\Phi}_C(K,t), \mathcal{U}^{\Phi}_D(K,t)\}$. As we have in the Proposition \ref{p2.13} part (1), if $\chi_W\in \mathcal{U}^{\Phi}(K,t)$, then $\chi_F\in \mathcal{U}^{\Phi}(HK,t)$, for all $F\subseteq W$ with $\mu(F)>0$. Also by the Proposition \ref{pp2.11}, first part, $\chi_f\in \mathcal{U}^{\Phi}(HK,t)$, for all $F\subseteq \varphi^i(W)$, with $\mu(F)>0$. Since by the Proposition \ref{p2.11}, $\mathcal{U}^{\Phi}(HK,t)$ is a linear space, then we have
 $$\{\sum^n_{i=0}a_i\chi_{F_i}: a_i\in \mathbb{C}, F_i\subseteq\varphi^{j_i}(W), j_i\in\mathbb{Z}, \mu(F_i)>0, F_i\cap F_{i'}=\emptyset, i\neq i'\}\subseteq \mathcal{U}^{\Phi}(HK,t).$$
By the fact that simple functions are dense in Orlicz space $L^{\Phi}(\mu)$, in case $\Phi\in\Delta_2$, the assumption $X=\cup^{+\infty}_{j=-\infty}\varphi^j(W)$ and by the Proposition \ref{p2.14}, part (1), we get the result.\\
(2) By the Proposition \ref{p2.13} part (2), if $\chi_W\in \mathcal{U}^{\Phi}_{GH-}(K,t)$, then $\chi_F\in \mathcal{U}^{\Phi}_{GH-}(HK,t)$, for all $F\subseteq W$ with $\mu(F)>0$. Also, by the Proposition \ref{pp2.11}, first part, $\chi_f\in \mathcal{U}^{\Phi}_{GH-}(HK,t)$, for all $F\subseteq \varphi^i(W)$, with $\mu(F)>0$ and $i\geq0$. Similar to part (1), by the linearity of $\mathcal{U}^{\Phi}_{GH-}(K,t)$ (Proposition \ref{p2.11}), we get that
$$\{\sum^n_{i=0}a_i\chi_{F_i}: a_i\in \mathbb{C}, F_i\subseteq\varphi^{j_i}(W), j_i<0, \mu(F_i)>0, F_i\cap F_{i'}=\emptyset, i\neq i'\}\subseteq \mathcal{U}^{\Phi}_{GH-}(HK,t)$$
Hence by the Proposition \ref{p2.14}, part (3), we get the result.\\
(3) The proof is similar to the proof of part(2).
\end{proof}
\begin{thm}\label{t2.16}
If $(X,\mathcal{F}, \mu,\varphi)$ is a dissipative system of bounded distortion generated by $W$, then the following hold.
\begin{enumerate}
  \item If the condition \ref{hc} is satisfied, then $C_{\varphi}$ is a proper contraction under an equivalent norm, i.e., $r(C_{\varphi})<1$.
  \item If the condition \ref{hd} holds, then $C_{\varphi}$ is a proper dilation under an equivalent norm, i.e., $r(C^{-1}_{\varphi})<1$.
  \item If the condition \ref{gh} is satisfied, then $C_{\varphi}$ is a generalized hyperbolic operator.
\end{enumerate}
So $C_{\varphi}$ has the shadowing property in all three cases.
\end{thm}
\begin{proof}
(1) If the condition \ref{hc} holds, then by the Proposition \ref{pp2.10} we have $\chi_{w}\in \mathcal{U}^{\Phi}_{C}(K,t)$, for some $K>0$ and $0<t<1$. On the other hand by the Proposition \ref{p2.14}, $L^{\Phi}(\mu)=\mathcal{U}^{\Phi}_C(HK,t)$. Hence by definition, for every $f\in L^{\Phi}(\mu)$, and $\forall n\in \mathbb{N}$
$$\sup_{k\in \mathbb{Z}}\frac{N_{\Phi}(C^k_{\varphi^{-1}}f)}{N_{\Phi}(C^{k+n}_{\varphi^{-1}}(f))}\leq HKt^n.$$
Especially for all $n\in \mathbb{N}$ we have
$\sup_{k\in \mathbb{Z}}\frac{N_{\Phi}(C^n_{\varphi}f)}{N_{\Phi}(f)}\leq HKt^n$. Hence
$$\left(\frac{N_{\Phi}(C^n_{\varphi}f)}{N_{\Phi}(f)}\right)^{\frac{1}{n}}\leq (HK)^{\frac{1}{n}}t, \ \ \ \ \ \forall n\in \mathbb{N}$$
and so $r(C_{\varphi})=\lim_{n\rightarrow \infty}\|C^n_{\varphi}\|^{\frac{1}{n}}\leq t<1$. Therefore we get the result. Similar to the part (1), we can prove (2).\\
(3) If the condition \ref{gh} holds, then by the Proposition \ref{pp2.10} we have $\chi_{w}\in \mathcal{U}^{\Phi}_{GH-}(K,t)$ and $C_{\varphi}(\chi_{w})\in \mathcal{U}^{\Phi}_{GH+}(K,t)$, for some $K>0$ and $0<t<1$. On the other hand by the Proposition \ref{p2.14}, $\mathcal{U}^{\Phi}_{GH+}(HK,t)=L^{\Phi}_+$ and $\mathcal{U}^{\Phi}_{GH-}(HK,t)=L^{\Phi}_-$. So by definitions of $\mathcal{U}^{\Phi}_{GH+}(HK,t)$ and $\mathcal{U}^{\Phi}_{GH-}(HK,t)$, we have for every $f\in L^{\Phi}_+$,
$$
\sup_{k\in -\mathbb{N}_0}\frac{N_{\Phi}(C^{k-n}_{\varphi^{-1}}(f))}{N_{\Phi}(C^{k}_{\varphi^{-1}}(f))}\leq HKt^n, \ \ \ \ \ \forall n\in \mathbb{N}$$
 and for every $g\in L^{\Phi}_-$
$$
\inf_{k\in \mathbb{N}_0}\frac{N_{\Phi}(C^k_{\varphi^{-1}}(g))}{N_{\Phi}(C^{k+n}_{\varphi^{-1}}(g))}\geq HK\frac{1}{t^n}, \ \ \ \ \ \forall n\in \mathbb{N}.$$
So for every $n\in \mathbb{N}$ and $k\in-\mathbb{N}_0$
$$\frac{N_{\Phi}(C^{k-n}_{\varphi^{-1}}(f))}{N_{\Phi}(C^{k}_{\varphi^{-1}}(f))}\leq HKt^n$$
and for all $k\in\mathbb{N}_0$
$$\frac{N_{\Phi}(C^k_{\varphi^{-1}}(g))}{N_{\Phi}(C^{k+n}_{\varphi^{-1}}(g))}\geq HK\frac{1}{t^n}  \Rightarrow \sup_{k\in \mathbb{N}_0}\frac{N_{\Phi}(C^{k+n}_{\varphi^{-1}}(g))}{N_{\Phi}(C^{k}_{\varphi^{-1}}(g))}\leq \frac{1}{HK}t^n$$
Hence by taking $k=0$ in both cases we get that for all $n\in \mathbb{N}$,

$$\frac{N_{\Phi}(C^{n}_{\varphi}(f))}{N_{\Phi}(f)}\leq HKt^n \ \ \ \text{and}  \ \ \ \frac{N_{\Phi}(C^{n}_{\varphi^{-1}}(g))}{N_{\Phi}(g)}\leq \frac{1}{HK}t^n.$$
By these observations we get that $r(C_{\varphi}|_{L^{\Phi}_+})\leq t<1$ and $r(C_{\varphi^{-1}}|_{L^{\Phi}_-})\leq t<1$. This completes the proof.

\end{proof}
Here we recall definitions of factor map.
\begin{defn}
 Let $(X,S)$ and $(Y,T)$ be two linear dynamical systems. We say that $T$ is a factor of $S$ if there exists a linear, continuous and surjective map $\Pi:X\rightarrow Y$ such that $\Pi\circ S=T\circ \Pi$. The map $\Pi$ is called the factor map. Moreover, we say $\Pi$ admits a bounded selector if there exists $L>0$ such that
 $$\forall y\in Y,  \ \ \exists x\in \Pi^{-1}(\{y\}), \ \ \ \ \text{with}  \ \ \ \ \|x\|\leq L\|y\|.$$
\end{defn}
\begin{lem}\label{l2.18}
Suppose that $(X, \mathcal{F}, \mu, \varphi)$ has bounded distortion, $\Phi\in \Delta'$ and $h_k=\frac{d\mu\circ\varphi^{-k}}{d\mu}$, the Radon-Nikodym derivative of $d\mu\circ\varphi^{-k}$ with respect to $\mu$. Let $B_w$ be the back-ward shift on Orlicz sequence space $l^{\Phi}(\mathbb{Z})$ with weights
$$w_k=w_k=\frac{N_{\Phi}(C^{k-1}_{\varphi^{-1}}(\chi_W))}{N_{\Phi}(C^{k}_{\varphi^{-1}}(\chi_W))}.$$
Then $B_w$ is a factor of the map $C_{\varphi}$ by a factor map $\Pi$ admitting a bounded selector.
\end{lem}
\begin{proof}
By invertibility of $C_{\varphi}$ we have $0<\inf_{n\in\mathbb{Z}}|w_n|\leq \sup_{n\in \mathbb{Z}}|w_n|<\infty$ and consequently we get that $B_w$ is invertible. Now we define the map
$$\Pi:L^{\Phi}(\mu)\rightarrow l^{\Phi}(\mathbb{Z}), \ \ \ \Pi(f)=\mathbf{x}=\{x_k\}_{k\in \mathbb{Z}},$$
for all $f\in L^{\Phi}(\mu)$, where $x_k=\frac{N_{\Phi}(C^k_{\varphi^{-1}}(\chi_W))}{N_{\Phi}(\chi_W)}\int_WC^k_{\varphi}(f)d\mu$. It is obvious that $\Pi$ is linear and $B_w\circ\Pi=\Pi\circ C_{\varphi}$. By our assumptions we have
$$\|h_k|_W\|_{\infty}\leq \sup_{B\subseteq W, \mu(B)\neq0}\frac{\mu(\varphi^{-k}(B))}{\mu(B)}.$$
So there exists $B\subseteq W$ with $\mu(B)\neq0$ such that $\frac{\mu(\varphi^{-k}(B))}{\mu(B)}\geq \|h_k|_W\|_{\infty}$. Since $\Phi\in \Delta'$, then we can find $N>0$ such that
$$\|h_k|_W\|_{\infty}\leq N\frac{\Phi^{-1}(\frac{1}{\mu(B)})}{\Phi^{-1}(\frac{1}{\mu(\varphi^{-k}(B))})}\leq N\sup_{B\subseteq W, \mu(B)\neq0}\frac{\Phi^{-1}(\frac{1}{\mu(B)})}{\Phi^{-1}(\frac{1}{\mu(\varphi^{-k}(B))})}, \ \ \ \forall k\in\mathbb{Z}$$
and therefore
$$\|h_k|_W\|_{\infty}\leq N\sup_{B\subseteq W, \mu(B)\neq0}\frac{\Phi^{-1}(\frac{1}{\mu(B)})}{\Phi^{-1}(\frac{1}{\mu(\varphi^{-k}(B))})}=N\sup_{B\subseteq W, \mu(B)\neq0}\frac{N_{\Phi}(C^k_{\varphi}(\chi_B))}{N_{\Phi}(\chi_B)}, \ \ \ \forall k\in\mathbb{Z}.$$

Moreover, we recall the Jensen's inequality that is $\Phi(\int_Xgd\mu)\leq\int_X\Phi(g)d\mu$, for every measurable and integrable function $g$ on $X$.
Now let $f\in L^{\Phi}(\mu)$. Then
\begin{align*}
\sum_{k\in\mathbb{Z}}\Phi\left(\frac{N_{\Phi}(C^k_{\varphi^{-1}}(\chi_W))}{NH.N_{\Phi}(f).N_{\Phi}(\chi_W)}\int_WC^k_{\varphi}(f)d\mu\right)&=\sum_{k\in\mathbb{Z}}\Phi\left(\int_W\frac{N_{\Phi}(C^k_{\varphi^{-1}}(\chi_W))}{NH.N_{\Phi}(f).N_{\Phi}(\chi_W)}C^k_{\varphi}(f)d\mu\right)\\
&\leq\sum_{k\in\mathbb{Z}}\int_W\Phi\left(\frac{N_{\Phi}(C^k_{\varphi^{-1}}(\chi_W))}{NH.N_{\Phi}(f).N_{\Phi}(\chi_W)}f\circ\varphi^k\right)d\mu\\
&=\sum_{k\in\mathbb{Z}}\int_{\varphi^k(W)}\Phi\left(\frac{N_{\Phi}(C^k_{\varphi^{-1}}(\chi_W))}{NH.N_{\Phi}(f).N_{\Phi}(\chi_W)}f\right)h_kd\mu\\
&\leq\sum_{k\in\mathbb{Z}}\int_{\varphi^k(W)}\frac{N_{\Phi}(C^k_{\varphi^{-1}}(\chi_W))}{NH.N_{\Phi}(\chi_W)}\Phi\left(\frac{f}{N_{\Phi}(f)}\right)\|h_k|_{\varphi^k(W)}\|_{\infty}d\mu\\
\end{align*}
Since
$$\|h_k|_W\|_{\infty}\leq N\sup_{\varphi^k(B), B\subseteq W}\frac{N_{\Phi}(\chi_B)}{N_{\Phi}(C^k_{\varphi^{-1}(\chi_B)})},$$
then by using the Proposition \ref{p2.6} we have
\begin{align*}
\sum_{k\in\mathbb{Z}}\Phi\left(\frac{N_{\Phi}(C^k_{\varphi^{-1}}(\chi_W))}{NH.N_{\Phi}(f).N_{\Phi}(\chi_W)}\int_WC^k_{\varphi}(f)d\mu\right)&\leq \sum_{k\in\mathbb{Z}}\int_{\varphi^k(W)}\frac{N_{\Phi}(C^k_{\varphi^{-1}}(\chi_W))}{H.N_{\Phi}(\chi_W)}\frac{HN_{\Phi}(\chi_W)}{N_{\Phi}(C^k_{\varphi^{-1}(\chi_W)})}\Phi\left(\frac{f}{N_{\Phi}(f)}\right)d\mu\\
&=\sum_{k\in\mathbb{Z}}\int_{\varphi^k(W)}\Phi\left(\frac{f}{N_{\Phi}(f)}\right)d\mu\\
&=\int_X\Phi\left(\frac{f}{N_{\Phi}(f)}\right)d\mu\leq 1.\\
\end{align*}
This means that $N_{\Phi}(\Pi(f))\leq NHN_{\Phi}(f)$ and so $\Pi$ is a bounded linear map. Here we prove that $\Pi$ admits a bounded selector. For $\mathbf{x}=\{x_k\}_{k\in \mathbb{Z}}\in l^{\Phi}(\mathbb{Z})$, we show that $N_{\Phi}(f)\leq N_{\Phi}(\mathbf{x})$, in which
$$f=\sum_{k\in \mathbb{Z}}\frac{x_k}{N_{\Phi}(C^{k}_{\varphi^{-1}}(\chi_W))}C^{k}_{\varphi^{-1}}(\chi_W).$$
Since $\Phi\in \Delta'$, then we have
\begin{align*}
\int_X\Phi(\frac{f}{N_{\Phi}(\mathbf{x})})d\mu&=\sum_{k\in\mathbb{Z}}\int_{\varphi^k(W)}\Phi(\frac{x_kC^{k}_{\varphi^{-1}}(\chi_W)}{N_{\Phi}(\mathbf{x}).N_{\Phi}(C^{k}_{\varphi^{-1}}(\chi_W))})d\mu\\
&\leq \sum_{k\in\mathbb{Z}}\int_{\varphi^k(W)}\Phi(\frac{x_k}{N_{\Phi}(\mathbf{x})})\Phi(\frac{C^{k}_{\varphi^{-1}}(\chi_W)}{N_{\Phi}(C^{k}_{\varphi^{-1}}(\chi_W))})d\mu\\
&=\sum_{k\in\mathbb{Z}}\Phi(\frac{x_k}{N_{\Phi}(\mathbf{x})})\int_{\varphi^k(W)}\Phi(\frac{C^{k}_{\varphi^{-1}}(\chi_W)}{N_{\Phi}(C^{k}_{\varphi^{-1}}(\chi_W))})d\mu\\
&\leq \sum_{k\in\mathbb{Z}}\Phi(\frac{x_k}{N_{\Phi}(\mathbf{x})})\leq1.
\end{align*}
This implies that $N_{\Phi}(f)\leq N_{\Phi}(\mathbf{x})$. It is easy to see that for every $k\in \mathbb{Z}$, $(\Pi(f))_k=x_k$. Therefore we get that $\Pi(f)=\mathbf{x}$.
This completes the proof.

\end{proof}

As is known, Orlicz spaces $L^{\Phi}(\mu)$ are an interpolation of the spaces $L^1(\mu)$ and $L^{\infty}(\mu)$.\cite{besh} Hence the concepts that are well defined and valid on $L^1(\mu)$ and $L^{\infty}(\mu)$, are well defined on Orlicz spaces too.
\begin{thm}\label{tt2.19}
Let $(X,\mathcal{F}, \mu, \varphi)$ be a dissipative system of bounded distortion generated by $W$ and $\Phi\in \Delta'$. If the composition operator $C_{\varphi}$ has the shadowing property, then one of the conditions \ref{hc}, \ref{hd} or \ref{gh} holds.
\end{thm}
\begin{proof}
By our assumptions and by the Lemma \ref{l2.18}, we have that $B_w:l^{\Phi}(\mathbb{Z})\rightarrow l^{\Phi}(\mathbb{Z})$ is a factor of $C_{\varphi}$, in which
$$w_k=\frac{N_{\Phi}(C^{k-1}_{\varphi^{-1}}(\chi_W))}{N_{\Phi}(C^{k}_{\varphi^{-1}}(\chi_W))}.$$
Also, as we saw in the Lemma \ref{l2.18}, the factor map $\Pi$ admits a bounded selector. Since $C_{\varphi}$ has the shadowing property, then by the Lemma 4.2.2 of \cite{ddm}, $B_w$ has shadowing property too.
It is easy to see that the Theorem 18 of \cite{bm} holds for Orlicz sequence spaces $l^{\Phi}(\mathbb{Z})$. So we have at least one of the conditions $A,B,C$ of Theorem 18 of cite{bm} holds. Since
$$w_kw_{k+1}...w_{k+n}=w_k\frac{N_{\Phi}(C^{k}_{\varphi^{-1}}(\chi_W))}{N_{\Phi}(C^{k+n}_{\varphi^{-1}}(\chi_W))}$$
and also $0<\inf|w_k|\leq\sup|w_k|<\infty$ (it comes from invertibility of $C_{\varphi}$ and $0<\mu(W)<\infty$), then we easily get that the conditions $A,B,C$ of Theorem 18 of \cite{bm} implies the conditions \ref{hc}, \ref{hd} and \ref{gh}, respectively.
\end{proof}
By Theorems \ref{t2.16} and \ref{tt2.19} we have the following characterization.
\begin{cor}\label{sc}
Let $(X,\mathcal{F}, \mu, \varphi)$ be a dissipative system of bounded distortion generated by $W$. Then the following are equivalent.
\begin{enumerate}
  \item The composition operator $C_{\varphi}$ has the shadowing property.
  \item One of conditions \ref{hc}, \ref{hd} or \ref{gh} holds.
\end{enumerate}
\end{cor}
As is known every generalized hyperbolic operator has shadowing property. Then by applying Theorems \ref{t2.16} and \ref{tt2.19} we have the following characterization.
\begin{cor}
Let $(X, \mathcal{F},\mu,\varphi)$ be a dissipative system of bounded distortion. Then the following are equivalent.
\begin{enumerate}
  \item The composition operator $C_{\varphi}$ is generalized hyperbolic.
  \item The composition operator $C_{\varphi}$ has the shadowing property.
\end{enumerate}
\end{cor}
Finally we provide an equivalent condition for composition operator $C_{\varphi}$ to has shadowing property based on Radon-Nikodym derivatives.
\begin{thm}
Let $(X, \mathcal{F},\mu,\varphi)$ be a dissipative system generated by $W$, $h_k=\frac{d\mu\circ\varphi^{-1}}{d\mu}$, $m_k={ess\,inf}_{x\in W} h_k(x)$, and $M_k={ess\,sup}_{x\in W} h_k(x)$. If the sequence $\{\frac{M_k}{m_k}\}_{k\in \mathbb{Z}}$ is bounded, then the following are equivalent.
\begin{enumerate}
  \item The composition operator $C_{\varphi}$ has the shadowing property.
  \item One of the following properties holds.
  \begin{equation}\label{rnc}
  \overline{\lim}_{n\rightarrow \infty}\sup_{k\in \mathbb{Z}}\left(\frac{M_k}{m_{k+n}}\right)^{\frac{1}{n}}<1
  \end{equation}
  \begin{equation}\label{rnd}
  \underline{\lim}_{n\rightarrow \infty}\inf_{k\in \mathbb{Z}}\left(\frac{M_k}{m_{k+n}}\right)^{\frac{1}{n}}>1
  \end{equation}
  \begin{equation}\label{rngh}
  \overline{\lim}_{n\rightarrow \infty}\sup_{k\in -\mathbb{N}_0}\left(\frac{M_{k-n}}{m_{k}}\right)^{\frac{1}{n}}<1 \ \ \ \text{and} \ \ \
  \underline{\lim}_{n\rightarrow \infty}\inf_{k\in \mathbb{N}_0}\left(\frac{M_k}{m_{k+n}}\right)^{\frac{1}{n}}>1.
  \end{equation}
  \end{enumerate}
Furthermore, conditions \ref{rnc}, \ref{rnd} imply that, under an equivalent norm, $C_{\varphi}$ is a proper contraction or a proper dilation, respectively. Condition \ref{rngh} implies that $C_{\varphi}$ is generalized hyperbolic.
\end{thm}
\begin{proof}
It is clear that $$m_i\mu(W)\leq\mu(\varphi^i(W))\leq M_i\mu(W), \ \ \ \ \forall i\in \mathbb{Z}.$$
Without lose of generality we can assume that $m_k<1$ and $M_k>1$. Hence by using some properties of Young's function $\Phi$ we have
$$\frac{1}{M_i}\Phi^{-1}(\frac{1}{\mu(W)})\leq \Phi^{-1}(\frac{1}{\mu(\varphi^i(W))})\leq \frac{1}{m_i}\Phi^{-1}(\frac{1}{\mu(W)})$$
 and so
$$\frac{m_{k}}{M_{k+n}}\leq \frac{\Phi^{-1}(\frac{1}{\mu(\varphi^{k+n}(W))})}{\Phi^{-1}(\frac{1}{\mu(\varphi^{k}(W))})}\leq \frac{M_{k}}{m_{k+n}}$$
and equivalently
Let $K$ be the bound of the bounded sequence $\{\frac{M_k}{m_k}\}_{k\in \mathbb{Z}}$. Then we get that
$$\frac{m_{k}}{M_{k+n}}\leq \frac{\Phi^{-1}(\frac{1}{\mu(\varphi^{k+n}(W))})}{\Phi^{-1}(\frac{1}{\mu(\varphi^{k}(W))})}\leq \frac{M_{k}}{m_{k+n}}\leq K^2\frac{m_{k}}{M_{k+n}}.$$
Hence
$$\overline{\lim}_{n\rightarrow \infty}\sup_{k\in \mathbb{Z}}\left(\frac{m_{k}}{M_{k+n}}\right)^{\frac{1}{n}}=\overline{\lim}_{n\rightarrow \infty}\sup_{k\in \mathbb{Z}} \left(\frac{\Phi^{-1}(\frac{1}{\mu(\varphi^{k+n}(W))})}{\Phi^{-1}(\frac{1}{\mu(\varphi^{k}(W))})}\right)^{\frac{1}{n}}=\overline{\lim}_{n\rightarrow \infty}\sup_{k\in \mathbb{Z}}\left(\frac{M_{k}}{m_{k+n}}\right)^{\frac{1}{n}}.$$
This implies that the condition \ref{hc} holds if and only if the condition \ref{rnc} holds. Similarly we get that the condition \ref{hd} holds if and only if the condition \ref{rnd} holds and the condition \ref{gh} holds if and only if the condition \ref{rngh} holds.

\end{proof}
\textbf{Declarations}\\
     \textbf{Conflict of interest.} None.\\
     \textbf{Acknowledgement.} My manuscript has no associate data.

\end{document}